\documentclass[11pt]{article}

\usepackage{amssymb, amsmath, amsthm, graphicx}
\usepackage[left=1in,top=1in,right=1in]{geometry}

\usepackage{subfigure}

      \theoremstyle{plain}
      \newtheorem{theorem}{Theorem}[section]
      \newtheorem{lemma}[theorem]{Lemma}

      \newtheorem{corollary}[theorem]{Corollary}
      \newtheorem{proposition}[theorem]{Proposition}

      \theoremstyle{definition}

      \theoremstyle{remark}

	\newcommand{\NN}{{\mathbb N}}

\def\ocn{\mbox{\rm odd-cr}}

\title{Disjoint edges in topological graphs\\ and the tangled-thrackle conjecture\thanks{Research on this paper began at the AIM workshop \emph{Exact Crossing Numbers (Palo Alto, CA, 2014)}.}}
\author{Andres J. Ruiz-Vargas\thanks{\'Ecole polytechnique f\'ed\'erale de Lausanne, Lausanne, Switzerland. Supported by Swiss National Science Foundation grant 200021-125287/1 and Swiss National
Science Foundation grant 200021-137574. 
Email: \texttt{andres.ruizvargas@epfl.ch}.}
\and
Andrew Suk\thanks{University of Illinois at Chicago, Chicago, IL, USA.  Supported by NSF grant DMS-1500153.  Email: {\tt suk@uic.edu}.}
\and
Csaba D. T\'oth\thanks{California State University Northridge, Los Angeles, CA, USA. Email: \texttt{cdtoth@acm.org}.}}

\begin{document}

\maketitle

\begin{abstract}
It is shown that for a constant $t\in \NN$, every simple topological graph on $n$ vertices has $O(n)$ edges if the graph has no two sets of $t$ edges such that every edge in one set is disjoint from all edges of the other set (i.e., the complement of the intersection graph of the edges is $K_{t,t}$-free). As an application, we settle the \emph{tangled-thrackle} conjecture formulated by Pach, Radoi\v{c}i\'c, and T\'oth:  Every $n$-vertex graph drawn in the plane such that every pair of edges have precisely one point in common, where this point is either a common endpoint, a crossing, or a point of tangency, has at most $O(n)$ edges.
\end{abstract}

\section{Introduction}

A \emph{topological graph} is a graph drawn in the plane such that its vertices are represented by distinct points and its edges are represented by Jordan arcs between the corresponding points
satisfying the following (nondegeneracy) conditions: (a) no edge intersects any vertex other
than its endpoints, (b) any two edges have only a finite number of interior points in common,
(c) no three edges have a common interior point, and (d) if two edges share an interior point, then they properly cross at that point~\cite{Pach07}. A topological graph is \emph{simple} if every pair of edges intersect in at most one point. Two edges of a topological graph \emph{cross} if their interiors share a point, and are \emph{disjoint} if they neither share a common vertex nor cross.

In 2005, Pach and T\'oth~\cite{pachtoth} conjectured that for every constant $t\geq 3$, an $n$-vertex simple topological graph has $O(n)$ edges if no $t$ edges are pairwise disjoint. They gave an upper bound of $|E(G)| \leq O(n\log^{4t-8} n)$ for all such graphs. Despite much attention over the last 10 years (see related results in \cite{fulek2,ps-ccmt-11,suk2,suk}), the conjecture is still open.

The condition that \emph{no $t$ edges are pairwise disjoint} means that the intersection graph of the edges (Jordan arcs) contains no anti-clique of size $t$, or equivalently the complement of the intersection graph of the edges is $K_t$-free. In this paper, we consider a stronger condition that the complement of the intersection graph of the edges is $K_{t,t}$-free, where $t\in \NN$ is a constant. This means that graph $G$ has no set of $t$ edges that are all disjoint from another set of $t$ edges. Since no such graph $G$ contains $2t$ pairwise disjoint edges, \cite{pachtoth} implies $|E(G)| \leq O(n\log^{8t-8}n)$. Our main result improves this upper bound to $O(n)$.

\begin{theorem}\label{thm:main}
Let $t\in \NN$ be a constant. The maximum number of edges in a simple topological graph with $n$ vertices that does not contain $t$ edges all disjoint from another set of $t$ edges is $O(n)$.
\end{theorem}

\paragraph{Application to thrackles.}
More than 50 years ago, Conway asked what is the maximum number of edges in an $n$-vertex \emph{thrackle}, that is, a simple topological graph $G$ in which every two edges intersect, either at a common endpoint or at a proper crossing~\cite{BMP05}. He conjectured that every $n$-vertex thrackle has at most $n$ edges. The first linear upper bound was obtained by Lov\'asz, Pach, and Szegedy~\cite{lovasz}, who showed that all such graphs have at most $2n$ edges. This upper bound was successively improved, and the current record is $|E(G)| \leq \frac{167}{117}n < 1.43n$ due to Fulek and Pach~\cite{fulek}.

As an application of Theorem~\ref{thm:main}, we prove the tangled-thrackle conjecture recently raised by Pach, Radoi\v{c}i\'c, and T\'oth~\cite{tangled}. A drawing of a graph $G$ is a \emph{tangled-thrackle} if it satisfies conditions (a)-(c) of topological graphs and every pair of edges have precisely one point in common: either a common endpoint, or a proper crossing, or a point of tangency. Note that such a drawing need not be a topological graph due to possible tangencies. Pach, Radoi\v{c}i\'c, and T\'oth~\cite{tangled} showed that every $n$-vertex tangled-thrackle has at most $O(n\log^{12}n)$ edges, and described a construction with at least $\lfloor 7n/6\rfloor$ edges. They conjectured that the upper bound can be improved to $O(n)$. Here, we settle this conjecture in the affirmative.

\begin{theorem}\label{thm:thrackle}
Every tangled-thrackle on $n$ vertices has $O(n)$ edges.
\end{theorem}

\section{Disjoint edges in topological graphs}

In this section, we prove Theorem \ref{thm:main}. We start with reviewing a few graph theoretic results used in our argument. The following is a classic result in extremal graph theory due to K\H{o}v\'ari, S\'os, and Tur\'an.

\begin{theorem}[Theorem 9.5 in \cite{pachag}]\label{kovari}
 Let $G = (V,E)$ be a graph that does not contain $K_{t,t}$ as a subgraph.
 Then $|E(G)| \leq c_1|V(G)|^{2 - 1/t}$, where $c_1$ is an absolute constant.
 \end{theorem}

Two edges in a graph are called \emph{independent} if they do not share an endpoint. We define the
\emph{odd-crossing number} $\ocn(G)$ of a graph $G$ to be the minimum number of unordered pairs of edges that are independent and cross an odd number of times over all topological drawings of $G$. The {\it bisection width} of a graph $G$, denoted by $b(G)$, is the smallest nonnegative integer such that there is a partition of the vertex set $V=V_1 \, \cup \, V_2$ with $\frac{1}{3} |V|\leq V_i\leq \frac{2}{3}|V|$ for $i=1,2$, and  $|E(V_1,V_2)|= b(G)$. The following result, due to Pach and T\'oth, relates the odd-crossing number of a graph to its bisection width.\footnote{Pach and T\'oth~\cite{pachtoth} defined the odd-crossing number of a graph $G$ to be the minimum number of pairs of edges that cross an odd number of times (over all drawings of $G$), including pairs of edges with a common endpoint.   However, since the number of pairs of edges with a common endpoint is at most $\sum_{i = 1}^n d_i^2$, this effects Theorem \ref{bisect} only by a constant factor.}

\begin{theorem}[\cite{pachtoth}]\label{bisect}
 There is an absolute constant $c_2$ such that if $G$ is a graph with $n$ vertices of vertex degrees $d_1,\ldots,d_n$, then $$b(G)\leq c_2\log n\sqrt{\ocn(G)+\sum_{i=1}^n d_i^2}.$$
\end{theorem}

We also rely on the result due to Pach and T\'oth~\cite{pachtoth} stated in the introduction.

\begin{theorem}[\cite{pachtoth}]\label{oldbound}
Let $G = (V,E)$ be an $n$-vertex simple topological graph, such that $G$ does not contain $t$ pairwise disjoint edges.  Then $|E(G)| \leq c_3n\log^{4t-8} n$, where $c_3$ is an absolute constant.
\end{theorem}

\paragraph{From disjoint edges to odd crossings.}
Using a combination of Theorems~\ref{kovari}--\ref{oldbound}, we establish the following lemma.

\begin{lemma}\label{lem:bisect2}
Let $G = (V,E)$ be a simple topological bipartite graph on $n$ vertices with vertex degrees $d_1,\ldots ,d_n$, such that $G$ does not contain a set of $t$ edges all disjoint from another set of $t$ edges.  Then
\begin{equation}\label{eq:bisect2}
b(G) \leq  c_4n^{1 - \frac{1}{2t}}\log^{8t -3}n + c_4\log n\sqrt{ \sum\limits_{i = 1}^{n}d_i^2},
\end{equation}
where $c_4$ is an absolute constant.
\end{lemma}
\begin{proof}
Since $G$ does not contain $2t$ pairwise disjoint edges, Theorem~\ref{oldbound} yields
\begin{equation}\label{bound1}
|E(G)| \leq c_3n\log^{8t - 8} n.
\end{equation}

We start by using the following redrawing idea of Pach and T\'oth \cite{pachtoth}.  Let $V_a$ and $V_b$ be the vertex classes of the bipartite graph $G$. Consider a simple curve $\gamma$ that decomposes the plane into two parts, containing all points in $V_a$ and $V_b$, respectively. By applying a suitable homeomorphism to the plane that maps $\gamma$ to a horizontal line, $G$ is deformed into a topological graph $G'$ such that (refer to Fig.~\ref{fig:redraw})
\begin{enumerate}\itemsep -1pt
\item the vertices in $V_a$ are above the line $y = 1$, the vertices in $V_b$ are below the line $y = 0$,
\item the part of any edge lying in the horizontal strip $0 \leq y \leq 1$ consists of vertical segments.
\end{enumerate}
\begin{figure}
  \centering
    \includegraphics[width=.7\textwidth]{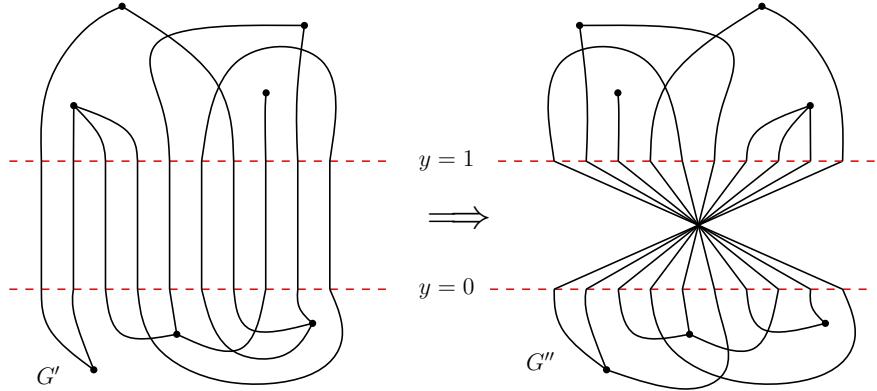}
          \caption{Redrawing procedure}\label{fig:redraw}
\end{figure}
Since a homeomorphism neither creates nor removes intersections between edges, $G$ and $G'$
are isomorphic and their edges have the same intersection pattern.

Next we transform $G'$ into a topological graph $G''$ by the following operations.
Reflect the part of $G'$ that lies above the $y = 1$ line about the $y$-axis. Replace the vertical
segments in the horizontal strip $0\leq y \leq 1$ by straight line segments that reconnect the corresponding pairs on the line $y = 0$ and $y = 1$, and perturb the segments if necessary to avoid triple intersections. It was shown in \cite{pachtoth} that if any two edges cross in $G$ (and $G'$), then they must cross an even number of times in $G''$. Indeed, suppose the edges $e_1$ and $e_2$ cross in $G$. Since $G$ is simple, they share exactly one point in common. Let $k_i$ denote the number of times edge $e_i$ crosses the horizontal strip for $i \in {1,2}$, and note that $k_i$ must be odd since the graph is bipartite. These $k_1+k_2$ segments within the strip pairwise cross in $G''$, creating ${k_1+k_2\choose 2}$ crossings. Since edge $e_i$ now crosses itself ${k_i\choose 2}$ times in $G''$, there are
\begin{equation}\label{count}
{k_1+k_2\choose 2}- {k_1\choose 2}  - {k_2 \choose 2} = k_1k_2
\end{equation}
crossings between edges $e_1$ and $e_2$ within the strip, which is odd when $k_1$ and $k_2$ are odd. Since $e_1$ and $e_2$ had one point in common outside the strip in both $G$ and $G''$, then $e_1$ and $e_2$ cross each other an even number
of times in $G''$. (Note that one can easily eliminate self-intersections by local modifications
around these crossings.)

Hence, the number of pairs of edges that are independent and cross an odd number of times in $G''$
is at most the number of disjoint pairs of edges in $G$, which is in turn, by Theorem~\ref{kovari} applied to the complement of the intersection graph of $G$, at most $c_1|E(G)|^{2 - 1/t}$ . Combined with (\ref{bound1}), we have
\begin{eqnarray}\label{ocn1}
 \ocn(G) & \leq & c_1(c_3n\log^{8t-8}n)^{2 - \frac{1}{t}} \nonumber\\
         & \leq & cn^{2-\frac{1}{t}} \log^{16t-8}n,       \nonumber
\end{eqnarray}
where $c$ is an absolute constant. Together with Theorem~\ref{bisect}, we have
\begin{eqnarray}
    b(G) & \leq & c_2\log n\sqrt{\left( cn^{2-\frac{1}{t}} \log^{16t-8}n\right) + \sum\limits_{i = 1}^{n}d_i^2} \nonumber\\
     & \leq & c_2\sqrt{c}n^{1 - \frac{1}{2t}}\log^{8t -3}n + c_2\log n\sqrt{ \sum_{i = 1}^{n}d_i^2}  \nonumber\\
     & \leq & c_4n^{1 - \frac{1}{2t}}\log^{8t -3}n + c_4\log n\sqrt{ \sum_{i = 1}^{n}d_i^2},  \nonumber
\end{eqnarray}
where $c_4$ is an absolute constant, as required.
\end{proof}

If the maximum degree of $G$ is relatively small, we obtain a sublinear bound on the
bisection width.

\begin{corollary}\label{cor:bisect3}
Let $G = (V,E)$ be simple topological bipartite graph on $n$ vertices with vertex degrees $d_1,\ldots ,d_n\leq n^{1/5}$ such that $G$ does not contain a set of $t$ edges all disjoint from another set of $t$ edges.  Then
\begin{equation}\label{eq:bisect3}
b(G) \leq  c_5n^{1 - \frac{1}{4t}},
\end{equation}
where $c_5$ is an absolute constant.
\end{corollary}
\begin{proof}
Substituting $d_i\leq n^{1/5}$ into~\eqref{eq:bisect2}, we have
\begin{eqnarray}
    b(G) & \leq & c_4 n^{1 - \frac{1}{2t}}\log^{8t -3}n + c_4\log n\sqrt{n\cdot n^{2/5}}  \nonumber\\
     & \leq & c_4 n^{1 - \frac{1}{2t}}\log^{8t -3}n + c_4n^{7/10}\log n \nonumber\\
     & \leq & c_5 n^{1 - \frac{1}{4t}},
\end{eqnarray}
for a sufficiently large constant $c_5$.
\end{proof}

\paragraph{Vertex splitting for topological graphs.}
Given a simple topological graph with $n$ vertices, we reduce the maximum degree below $n^{1/5}$ by a standard vertex splitting operation. Importantly, this operation can be performed such that it preserves the intersection pattern of the edges.

\begin{lemma}\label{lem:splitting}
Let $G$ be a simple topological graph with $n$ vertices and $m$ edges; and let $\Delta\geq 2m/n$.
Then there is a simple topological graph $G'$ with maximum degree at most $\Delta$, at most $n+2m/\Delta$ vertices, and precisely $m$ edges such that the intersection graph of its edges is isomorphic to that of $G$.
\end{lemma}
\begin{proof}
We successively split every vertex in $G$ whose degree exceeds $\Delta$ as follows. Refer to Fig.~\ref{fig:split}. Let $v$ be a vertex of degree $d(v) = d > \Delta$, and let $vu_1,vu_2,\ldots ,vu_d$ be the edges incident to $v$ in counterclockwise order. In a small neighborhood around $v$, replace $v$ by $\lceil d/\Delta\rceil$ new vertices, $v_1,\ldots , v_{\lceil d/\Delta\rceil}$ placed in counterclockwise order on a circle of small radius centered at $v$. Without introducing any new crossings, connect $u_j$ to $v_i$ if and only if $\Delta(i-1) < j \leq \Delta i$ for $j \in \{1,\ldots ,d\}$ and $i\in \{1,\ldots ,\Delta\}$. Finally, we do a local change within the small circle by moving each vertex $v_i$ across the circle, so that every edge incident to $v_i$ crosses all edges incident to $v_{i'}$, for all $i\neq i'$. As a result, any two edges incident to some vertex $\{v_1,\ldots , v_{\lceil d/\Delta\rceil}\}$ intersect precisely once: either at a common endpoint or at a crossing within the small circle centered at $v$.

 \begin{figure}
  \centering
    \includegraphics[width=.9\textwidth]{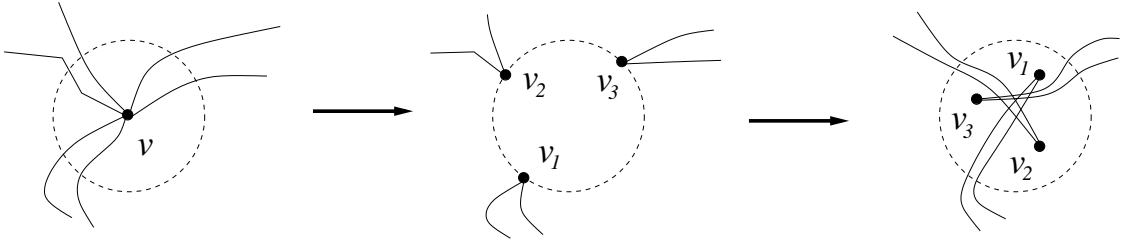}
          \caption{Splitting a vertex $v$ into new vertices $v_1,v_2,v_3$, such that each $v_i$ has degree at most $\Delta$.  Moreover, we do not introduce any disjoint pairs of edges and our new graph remains simple.}\label{fig:split}
\end{figure}

After applying this procedure to all vertices $G$, we obtain a simple topological graph $G'$ of maximum degree at most $\Delta$. By construction, $G'$ has $m$ edges, and the intersection pattern of the edges is the same as in $G$. The number of vertices in $G'$ is
$$|V(G')|
\leq \sum_{v \in V}\left\lceil\frac{d(v)}{\Delta}\right\rceil
\leq n + \sum_{v \in V} \frac{d(v)}{\Delta}
\leq n + \frac{2m}{\Delta},$$
as claimed.
\end{proof}

\paragraph{Putting things together.}
Since all graphs have a bipartite subgraph with at least half of its edges, Theorem~\ref{thm:main} immediately follows from the following.

\begin{theorem}\label{thm:main-}
Let $G$ be an $n$-vertex simple topological bipartite graph such that $G$ does not contain $t$ edges all disjoint from another set of $t$ edges. Then
\begin{equation}\label{eq:main-}
|E(G)| \leq c_6 n,
\end{equation}
where $c_6=c_6(t)$ depends only on $t$.
\end{theorem}
\begin{proof}
Let $t\in \NN$ be a constant.  We prove, by induction on $n$, that $|E(G)| \leq c_6 (n-n^{1-\frac{1}{7t}})$. The negative term will be needed in the induction step.  Let $n_0=n_0(t)$ be a sufficiently large constant (specified in \eqref{eq:1} and \eqref{eq:2} below) that depends only on $t$, and on the constants $c_3$ and $c_5$ defined in Theorem~\ref{oldbound} and Corollary~\ref{cor:bisect3}, respectively. Let $c_6$ be a sufficiently large constant such that $c_6\geq 2c_5$ and for every positive integer $n\leq n_0$, we have
\begin{equation}
c_3n\log^{8t-8} n \leq c_6(n-n^{1 - \frac{1}{7t}}),
\end{equation}

The choice of $n_0$ and $c_6$ ensures that \eqref{eq:main-} holds for all graphs
with at most $n_0$ vertices. Now consider an integer $n>n_0$, and
assume that \eqref{eq:main-} holds for all graphs with fewer than $n$ vertices.
Let $G$ be a simple topological bipartite graph with $n$ vertices such that $G$ does
not contain $t$ edges all disjoint from another set of $t$ edges.

By Theorem~\ref{oldbound}, $G$ has $m \leq c_3n\log^{8t-8} n$ edges. By Lemma~\ref{lem:splitting}, there is a simple topological graph $G'$ of maximum degree at most $\Delta=n^{1/5}$, $n'\leq n+2m/n^{1/5}$ vertices, and $m'=m$ edges, such that the intersection graph of its edges is isomorphic to that of $G$. Theorem~\ref{oldbound} implies that $n'\leq n+2m/n^{1/5}\leq n+2c_3n^{4/5}\log^{8t-8} n$. If $n\geq n_0$ for a sufficiently large constant $n_0$, then
\begin{equation}\label{eq:1}
n'\leq n+2c_3n^{4/5}\log^{8t-8} n\leq n+n^{5/6}.
\end{equation}

Since $G$ and $G'$ have the same number of edges, it is now enough to estimate $|E(G')|$. Note that $\Delta=n^{1/5}\leq (n')^{1/5}$, and by Corollary~\ref{cor:bisect3}, the bisection width of $G'$ is bounded by
$$
b(G') \leq c_5(n')^{1 - \frac{1}{4t}}
      \leq c_5(n+n^{5/6})^{1 - \frac{1}{4t}}
      \leq 2c_5 n^{1 - \frac{1}{4t}}.
$$

Partition the vertex set of $G'$ as $V'=V_1 \, \cup \, V_2$ with $\frac{1}{3} |V'|\leq V_i\leq \frac{2}{3}|V'|$ for $i=1,2$, such that $G'$ has $b(G')$ edges between $V_1$ and $V_2$. Denote
by $G_1$ and $G_2$ the subgraphs induced by $V_1$ and $V_2$, respectively. Put $n_1=|V_1|$ and $n_2=|V_2|$, where $n_1+n_2=n'\leq n+n^{5/6}$.

Note that both $G_1$ and $G_2$ are simple topological graphs that do not contain $t$ edges all disjoint from another set of $t$ edges. By the induction hypothesis, $|E(G_i)|\leq c_6(n_i-n_i^{1-1/7t})$ for $i=1,2$. The total number of edges in $G_1$ and $G_2$ is
\begin{eqnarray*}
 |E(G_1)|+|E(G_2)| &\leq & c_6(n_1-n_1^{1-\frac{1}{7t}}) + c_6(n_2-n_2^{1-\frac{1}{7t}})\nonumber\\
        & \leq & c_6 (n_1+n_2) -
            c_6(n_1^{1-\frac{1}{7t}}+n_2^{1-\frac{1}{7t}})\\
        & \leq & c_6(n')
          - c_6\left( \left(\frac{n_1}{n'}\right)^{1 - \frac{1}{7t}}+
          \left(\frac{n_2}{n'}\right)^{1 - \frac{1}{7t}}\right) (n')^{1 - \frac{1}{7t}}\\
        &\leq& c_6 (n+n^{5/6}) -c_6 \left( \left(\frac13\right)^{1 - \frac{1}{7t}}+\left(\frac23\right)^{1 - \frac{1}{7t}}\right)n^{1 - \frac{1}{7t}}\\
        &=& c_6 (n+n^{5/6}) -c_6 \alpha n^{1 - \frac{1}{7t}}\\
        &\leq& c_6 (n -n^{1 - \frac{1}{7t}})
             + c_6\left(n^{5/6} - (\alpha-1) n^{1 - \frac{1}{7t}}\right),
\end{eqnarray*}
where we write $\alpha = (1/3)^{1 - 1/7t}+(2/3)^{1 - 1/7t}$ for short.
Note that for every $t\in \NN$, we have $\alpha>1$.
Taking into account the edges between $V_1$ and $V_2$,
the total number of edges in $G'$ (and hence $G$) is
\begin{eqnarray}
|E(G')| &= & |E(G_1)|+|E(G_2)| + b(G') \nonumber\\
        &\leq& c_6 (n -n^{1 - \frac{1}{7t}}) + c_6\left(n^{5/6} - (\alpha-1) n^{1 - \frac{1}{7t}}\right)+  2c_5 n^{1 - \frac{1}{4t}}\nonumber\\
        &\leq& c_6 (n -n^{1 - \frac{1}{7t}})
             + c_6\left(n^{5/6} + n^{1 - \frac{1}{4t}}
             - (\alpha-1) n^{1 - \frac{1}{7t}}\right)\nonumber\\
        &\leq & c_6(n-n^{1-\frac{1}{7t}}), \label{eq:2}
\end{eqnarray}
where the last inequality holds for $n\geq n_0$ if $n_0$ is sufficiently large (independent of $c_6$). This completes the induction step, hence the proof of Theorem~\ref{thm:main-}.
\end{proof}

\section{Application: the tangled-thrackle conjecture}

Let $G$ be tangled-thrackle with $n$ vertices. By slightly modifying the edges (i.e., Jordan arcs) near the points of tangencies, we obtain a simple topological graph $\tilde{G}$ with the same number of vertices and edges such that every pair of tangent edges in $G$ become disjoint in $\tilde{G}$ and all other intersection points between edges remain the same. In order to show that $|E(G)| \leq O(n)$, invoking Theorem \ref{thm:main}, it suffices to prove the following.

\begin{lemma}\label{100e}
For every tangled-thrackle $G$, the simple topological graph $\tilde{G}$ does not contain a set of 200 edges all disjoint from another set of 200 edges.
\end{lemma}
Before proving the lemma, we briefly review the concept of
Davenport-Schinzel sequences and arrangements of pseudo-segments.

A finite sequence $U = (u_1,\ldots , u_t)$ of symbols over a finite alphabet is a \emph{Davenport-Schinzel sequence of order $s$} if it satisfies the following two properties:
\begin{itemize}
\item no two consecutive symbols in the sequence are equal to each other;
\item for any two distinct letters of the alphabet, $a$ and $b$, the sequence does not contain a (not necessarily consecutive) subsequence $(a,b,a,\ldots, b,a)$ consisting of $s + 2$ symbols alternating between $a$ and $b$.
\end{itemize}

The maximum length of a Davenport-Schinzel sequence of order $s$ over an alphabet of size $n$ is denoted $\lambda_s(n)$. Sharp asymptotic bounds for $\lambda_s(n)$ were obtained by Nivasch \cite{nivasch} and Pettie \cite{seth}. However, to avoid the constants hidden in the big-Oh notation, we use simpler explicit bounds. Specifically, we use the following upper bound for $\lambda_3(n)$.

\begin{theorem}[see Proposition 7.1.1 in~\cite{matousek}]
$\lambda_3(n) < 2n \ln n+3n$.
\end{theorem}

A set $\mathcal L$ of $m$ Jordan arcs in the plane is called an \emph{arrangement of pseudo-segments} if each pair of arcs intersects in at most one point (at an endpoint, a crossing, or a point of tangency), and no three arcs have a common interior point. An arrangement of pseudo-segments naturally defines a plane graph: The \emph{vertices} of the arrangement are the endpoints and the intersection points of the Jordan arcs, and the \emph{edges} are the portions of the Jordan arcs between consecutive vertices. The \emph{faces} of the arrangement are the connected components of the complement of the union of the Jordan arcs. The vertices and edges are said to be \emph{incident to} a face if they are contained in the (topological) closure of that face. The following theorem is a particular case of Theorem 5.3 of \cite{micha}.

\begin{theorem}[see~\cite{micha}]\label{facesize}
Let $\mathcal L$ be an arrangement of $m$ pseudo-segments and $F$ be a face of $\mathcal L$.
Then number of edges incident to $F$ is at most $\lambda_3(2m)$.
\end{theorem}

\begin{figure}
  \centering
    \includegraphics[width=.7\textwidth]{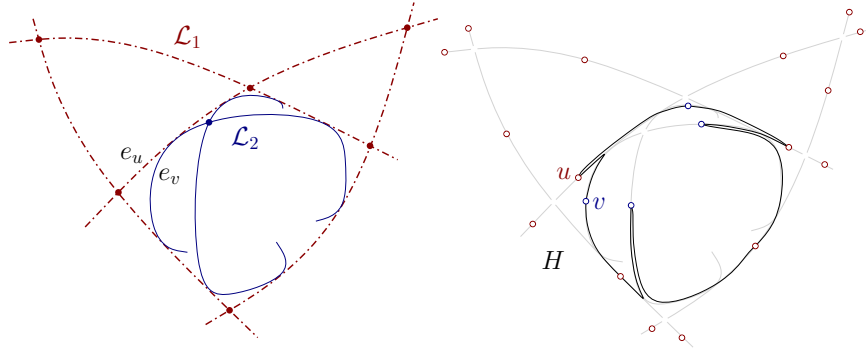}
          \caption{Constructing $H$ from $\mathcal L_1$ and $\mathcal L_2$.}\label{fig:tangencies}
\end{figure}

\begin{lemma}\label{lem:connected}
Let $\mathcal L_1\cup \mathcal L_2$ be an arrangement of pseudo-segments such that every arc in $\mathcal L_1$ is tangent to all arcs in $\mathcal L_2$; and $\mathcal L_1$ and $\mathcal L_2$ each form a connected arrangement. Then $\mathcal L_1$ or $\mathcal L_2$ contains at most 200 arcs. \end{lemma}
\begin{proof} Refer to Figure \ref{fig:tangencies}.
Suppose to the contrary, that both $\mathcal L_1$ and $\mathcal L_2$ contain at least 200 arcs.
Without loss of generality, we may assume $|\mathcal L_1|=|\mathcal L_2|=200$.
Since no arc in $\mathcal L_1$ crosses any arc in $\mathcal L_2$, the arrangement $\mathcal L_1$ lies in the closure of a single face $F_2$ of the arrangement $\mathcal L_2$, and vice versa $\mathcal L_2$
lies in the closure of a single face $F_1$ of the arrangement $\mathcal L_1$. We construct a plane graph $H$ representing the tangencies between the edges of the two arrangements: place a vertex on the relative interior of each edge of $\mathcal L_1$ incident to $F_1$ and each edge of $ \mathcal L_2$ incident to $F_2$. Join two vertices, $v$ and $u$, by an edge iff their corresponding edges of the arrangements, $e_u$ and $e_v$, are tangent to each other. To see that $H$ is indeed planar, note that each edge $uv$ can be drawn closely following the arcs $e_u$ and $e_v$ to their intersection point in such a way that $H$ has no crossings. As every arc in $\mathcal L_1$ is tangent to all arcs in $\mathcal L_2$, the graph $H$ has exactly $200^2$ edges. By Theorem \ref{facesize}, $H$ has at most $2\lambda_3(400)$ vertices.  However,
$$ \frac{|E(H)|}{|V(H)|}
\geq \frac{200^2}{2\lambda_3(400)}
>\frac{200^2}{4\cdot 400 \ln 400 +6\cdot 400}
> 3.3,$$
which contradicts Euler's formula.
\end{proof}

Note that Lemma~\ref{lem:connected} easily generalizes to the case when the arrangements $\mathcal L_1$ and $\mathcal L_2$ are not necessarily connected but they have the property that every pseudo-segment in $\mathcal L_1$ lies in the same face of $\mathcal L_2$, and vice versa.

It is now easy to see that Lemma~\ref{100e} follows directly from Lemma~\ref{lem:connected}: if 200 edges of the simple topological graph $\tilde{G}$ are disjoint from another set of 200 edges of  $\tilde{G}$, then a set of the corresponding 200 edges of the tangled-thrackle $G$ are tangent to corresponding other set of 200 edges of $G$.

\medskip

\noindent\textbf{Proof of Theorem \ref{thm:thrackle}:}  The statement follows by combining Theorem \ref{thm:main} and Lemma \ref{100e}.$\hfill\square$

\medskip

We now show an analogue of Lemma \ref{lem:connected} where we drop the condition that the arrangements $\mathcal L_1$ and $\mathcal L_2$ are connected.  We find this interesting for its own sake.
\begin{proposition}\label{lem:disconnected}
Let $\mathcal L_1\cup \mathcal L_2$ be an arrangements of pseudo-segments such that every arc in $\mathcal L_1$ is tangent to all arcs in $\mathcal L_2$. Then $\mathcal L_1$ or $\mathcal L_2$ contains at most 800 arcs.
\end{proposition}
\begin{proof}
Suppose to the contrary, that both $\mathcal L_1$ and $\mathcal L_2$ contain at least 800 arcs.
Without loss of generality, we may assume $|\mathcal L_1|=|\mathcal L_2|=800$.
The difference from Lemma \ref{100e} is that the arrangement $\mathcal L_1$ or $\mathcal L_2$ may not be connected, and so the arcs in $\mathcal L_1$ could be distributed in several faces of the arrangement $\mathcal L_2$.

Consider an arc $\ell\in \mathcal L_1$, and denote by $s$ the number of other arcs in $\mathcal L_1$ that intersect $\ell$, where $0\leq s\leq 799$. Then $\ell$ contains precisely $s+1$ edges of the arrangement $\mathcal L_1$. Partition $\ell$ into two Jordan arcs: $\ell_a\subset\ell$ consists of the first $\lfloor (s+1)/2\rfloor$ edges along $\ell$, and $\ell_b=\ell\setminus \ell_a$.
Recall that $\ell$ is tangent to all 800 arcs in $\mathcal L_2$. By the pigeonhole principle, we may assume w.l.o.g. that $\ell_a$ is tangent to at least 400 arcs in $\mathcal L_2$. Let $\mathcal L_2'\subset \mathcal L_2$ be a set of 400 arcs in $\mathcal L_2$ that are tangent to $\ell_a$.
By construction, $\ell_a$ intersects at most $\lceil s/2\rceil\leq 400$ arcs in $\mathcal L_1$. Consequently, there is a set $\mathcal L_1'\subset \mathcal L_1$ of 400 arcs in $\mathcal L_1$
that do \emph{not} intersect $\ell_a$. Observe that $\ell_a$ lies in a single face of the arrangement
$\mathcal L_1'$. Since every arc in $\mathcal L_2'$ intersects $\ell_a$, all arcs in $\mathcal L_2'$ lie in the same face of the arrangement $\mathcal L_1'$.

Applying the same procedure to $\mathcal L_2'$ and $\mathcal L_1'$ we may now find sets $\mathcal L_2''\subset \mathcal L_2'$ and $\mathcal L_1''\subset \mathcal L_1'$ each of size 200 such that all the arcs of $\mathcal L_1''$ lie in the same face of the arrangement $\mathcal L_2''$

We have found subsets $\mathcal L_1''\subset \mathcal L_1$ and $\mathcal L_2''\subset \mathcal L_2$
of size $|\mathcal L_1''|=|\mathcal L_2''|=200$ such that every arc in $\mathcal L_1''$ is tangent to all arcs in $\mathcal L_2''$; and all arcs of $\mathcal L_1'$ lie in the same face of $\mathcal L_2'$ and vice versa. This contradicts Lemma~\ref{lem:connected} and the remark following its proof.
\end{proof}

\section{Concluding remarks}

1. We showed that for every integer $t$, the maximum number of edges in a simple topological graph with $n$ vertices that does not contain $t$ edges all disjoint from another set of $t$ edges is $cn$, where $c = c(t)$.  A careful analysis of the proof shows that $c = 2^{O(t\log t)}$.  It would be interesting to see if one could improve the upper bound on $c$ to $O(t)$.

\medskip

\noindent 2.  We suspect that the bounds of 200 and 800 in Lemma~\ref{lem:connected} and Proposition~\ref{lem:disconnected} are not optimal. Since any constant bound yields a linear upper bound for the number of edges in  tangled-thrackles, we have not optimized these values. However, finding the best possible constants, or shorter proofs for some arbitrary constant bounds, would be of interest.

\paragraph{Acknowledgments.} We would like to thank G\'abor Tardos for suggestions that helped simplify the main proof.  We also thank the anonymous referees for helpful comments.

\end{document}